\documentclass[10pt]{article}
\usepackage{graphicx}
\usepackage[all]{xy}
\input xy
\usepackage{amsmath}
\usepackage{amsfonts}
\usepackage{amssymb}
\usepackage{url}
\usepackage{multirow}

\usepackage[mathcal]{eucal}
\usepackage{enumerate}
\usepackage{enumitem}
\usepackage{anysize}
\usepackage[english]{babel}
\usepackage{hyperref}
\usepackage{accents}

\usepackage{float}

\newtheorem{remark}{Remark}

\newtheorem{theorem}{Theorem}
\newtheorem{lemma}{Lemma}
\newtheorem{corollary}[theorem]{Corollary}

\newenvironment{proof}{\noindent{\bf Proof.}}%
{\hspace*{\fill}$\Box$\par\vspace{4mm}}

\newcommand{\pr} {{\rm Pr}}
\newcommand{\pra}[1] {\pr\left\{#1\right\}}

\newcommand{\E} {\mathbb{E}}
\newcommand{\Euc} {\mathcal{E}}
\newcommand{\bl}[1] {\mathfrak{b}(#1)}
\newcommand{\al}[1] {\mathfrak{a}(#1)}
\newcommand{\Ea}[1] {\E\left(#1\right)}

\newcommand{\che} {\v{C}eby\v{s}\"{e}v}
\newcommand{\R} {\mathbb{R}}
\newcommand{\Rn} {\mathcal{R}^n}

\newcommand{\nmaha}[1] {\|#1\|_{\Sigma^{-1}}}

\newcommand{\inp}[1] {\langle #1\rangle}

\newcommand{\virg}[1] {``#1''}

\newcommand{\mybreak} {\par\vspace{2mm}\noindent}

\def\cadre{$$\vcenter\bgroup\advance\hsize by -2em\noindent
	\refstepcounter{equation}(\theequation)~\ignorespaces}
\makeatletter
\def\endcadre{\egroup\eqno$$\global\@ignoretrue}

\begin{document}

\title{Cantelli's bounds for generalized tail inequalities in Euclidean spaces.}

\author{Nicola Apollonio\footnote{Istituto per le Applicazioni del
		Calcolo, ``Mauro Picone'', Consiglio Nazionale delle Ricerche, Via dei Taurini 19, 00185 Roma, Italy.
		\texttt{nicola.apollonio@cnr.it}.}
}

\date{}

\maketitle	




\begin{abstract}
 Let $X$ be a centered random vector in a finite dimensional real inner product space $\Euc$. For a subset $C$ of the ambient vector space $V$ of $\Euc$ and $x,\,y\in V$, write $x\preceq_C y$ if $y-x\in C$. When $C$ is a closed convex cone in $\Euc$, then $\preceq_C$ is a pre-order on $V$, whereas if $C$ is a proper cone in $\Euc$, then $\preceq_C$ is actually a partial order on $V$. In this paper we give sharp Cantelli's type inequalities for generalized tail probabilities like $\pra{X\succeq_C b}$ for $b\in V$. These inequalities are obtained by \virg{scalarizing} $X\succeq_C b$ via cone duality and then by minimizing the classical univariate Cantelli's bound over the scalarized inequalities.
\end{abstract}
\mybreak
\mybreak




\section{Introduction}\label{sec:intro}
Let $Y$ be a random variable with finite mean $\mu$ and variance $\sigma^2$. Hence, the random variable $X-\mu$ is centered and has the same variance as $Y$.\,For a positive real number $b$, the celebrated Cantelli inequality---also known as \emph{one sided \che-inequality}---reads as
\begin{equation}\label{eq:can_cla}
\pra{X\geq b}\leq \frac{\sigma^2}{b^2+\sigma^2}.
\end{equation}
Both Cantelli's inequality and the classical \che's inequality can be (and in fact have been) extended in several ways \cite{mar_olk, vbc} to a random vector $X=(X_1,\ldots,X_n)'$ in $\R^n$---here and throughout the rest of the paper, $u'$ denotes the transpose of column vector $u\in\mathbb{R}^n$---. As shown in \cite{mar_olk} and in \cite{vbc}, there is a standard recipe that yields such extensions: consider a random vector $X$ supported by a subset $S$ of $\R^n$. Let $T$ be a subset of $S$, and $f: S\rightarrow \R$ be such that $f(x)\geq 0$ for all $x\in S$ and $f(x)\geq 1$ for all $x\in T$. Then, with $\mathbf{1}_T(\cdot)$ denoting the indicator of set $T$ over $S$, one has $f\geq \mathbf{1}_T$ and  
$$\Ea{f(X)}\geq \Ea{f(X)\mathbf{1}_T(X)}\geq \Ea{\mathbf{1}_T(X)}=\pra{X\in T}.$$  
This technique is essentially \virg{Markov inequality}. By taking $f$ in the family $\{f_u \ |\ f_u:\R^n\rightarrow \R^n_+,\,u\in \R^n\}$ where $f_u(x)=\frac{(u'x+u'\Sigma u)^2}{1+u'\Sigma u}$, and minimizing for $u$ under the constraint $f\geq \mathbf{1}_T$,  Marshall and Olkin obtained the following strong and general result.
\begin{theorem}[Marshall and Olkin \cite{mar_olk}]\label{thm:mar_olk}
Let $T$ be a closed convex set in $\R^n$ not containing the origin. If $X$ is a centered random vector of $\R^n$ with positive definite covariance matrix $\Sigma$, then 
\begin{equation}\label{eq:mo_result}
\pra{X\in T}\leq \inf_{\substack{u\in \R^n\\u'x\geq 1,\,\forall x\in T}}\frac{u'\Sigma u}{1+u'\Sigma u}.
\end{equation}
Furthermore, the inequality is sharp, in the sense that there exists a centered random vector $X_0$ whose support contains $T$ and whose covariance matrix is $\Sigma$, such that the inequality is attained as equality. 
\end{theorem}
Notice that Cantelli's inequality \eqref{eq:can_cla} follows from inequality \eqref{eq:mo_result}, after dividing the univariate random variable $X$ by the positive threshold $b$ and observing that the variance of $X/b$ is $\sigma^2/b^2$. The minimization problem on the right-hand side of \eqref{eq:mo_result} is solved by minimizing the quadratic form $u'\Sigma u$ over the same set. Thus, this is a convex minimization problem which can be solved by the techniques described in \cite{boyd}. As proved in \cite{mar_olk}, the infimum in \eqref{eq:mo_result} is attained. The function $f_{\hat{u}}$ corresponding to the vector $\hat{u}$ attaining the infimum in  inequality \eqref{eq:mo_result}, can be seen as a kind of envelope of a given shape (quadratic, in the present case) for the probability on the right-hand side. The same inequality can be interpreted as follows: first we approximate $T$ linearly inside the probability; this approximation yields a family of linear inequalities, each of which is the tail of a scalar random variable; we use Cantelli's inequality \eqref{eq:can_cla} to bound each of these tails, and finally we choose the tightest one. Let us describe this process for a non-empty arbitrary subset $T$ of $\R^n$: let $\bl{T}=\{u\in \R^n \ |\ u'x\geq 1,\quad \forall x\in T\}$---observe that $\bl{T}$ is always a closed convex set regardless of the argument $T$ (see Section \ref{sec:prel} for more details)---; since, by definition,  $T\subseteq \bl{\bl{T}}$ and $x\in \bl{\bl{T}}\Leftrightarrow u'x\geq 1,\, \forall u\in \bl{T}$, it follows that if $\bl{T}\not=\emptyset$, then
$$\pra{X\in T}\leq \pra{X\in \bl{\bl{T}}}=\pra{u'X\geq 1,\, \forall u\in \bl{T}},$$
where $X$ is a centered random vector with positive definite covariance matrix $\Sigma$. Now, for all $u\in \bl{T}$, the random variable $u'X$ is a centered random variable with variance $u'\Sigma u$. Thus, in view of Cantelli's inequality, for all $u\in \bl{T}$ it holds that $$\pra{X\in T}\leq \frac{u'\Sigma u}{1+u'\Sigma u}.$$
Hence, when $T$ is convex and $\bl{T}$ is non-empty, after taking the infimum over $\bl{T}$, we recover \eqref{eq:mo_result} from another perspective. In this paper we show, via cone duality, that the same result holds in every finite dimensional Euclidean space $\Euc$. Moreover, when $T$ is of the form $b+C$, with $C$ being a convex cone in $\Euc$, we provide a specialized Cantelli's bound which is sharp. When, furthermore, $C$ has a non-empty interior, then $C$ induces a preorder $\succeq_C$ on the ambient space of $\Euc$ such that the event $(X\in T)$ can be written as $(X\succeq_C b)$ and can be interpreted as a generalized tail inequality (we recover classical tail inequalities when $C$ is the non-negative orthant). Hence Cantelli's inequality naturally extends to generalized tail inequalities in finite dimensional Euclidean spaces. In addition, for special choices of the cone $C$ and the threshold $b$, Cantelli's inequality for generalized tails has a particular simple expression in terms of certain norms of $b$. In the case of $\R^n$, with the standard inner product, if $\Sigma^{-1}$ is a non-negative matrix and $b\geq 0$, such an expression reads as
$$\pra{X\geq b}\leq (1+\|b\|_M^2)^{-1}$$
where $\|b\|_M=\sqrt{b'\Sigma^{-1}b}$ is the \emph{Mahalanobis norm} of $b$. Notice that inequality \eqref{eq:can_cla} specializes to the inequality above, after dividing numerator and denominator by $\sigma^2$. This yields yet another extension. 

The rest of the paper goes as follows. In Section \ref{sec:prel} we develop the machinery to state and prove the main results provided in Section \ref{sec:main} with the help of three examples involving the least eigenvalue of a random symmetric matrix, statistical hypotheses testing, and the feasibility of a system of linear inequalities with coefficients in $\{0,1\}$.

\section{Preparatory results}\label{sec:prel}
Finite dimensional Euclidean spaces, namely, finite dimensional real vector spaces $V$ equipped with an inner product $\inp{\cdot,\cdot}$, will be denoted by calligraphic letters, such as $\Euc$. 
The symbol $\Rn$ stands for the Euclidean space $(\R^n,\inp{\cdot,\cdot})$, where $\inp{\cdot,\cdot}$ is the standard dot product in $\R^n$. In what follows, when we speak of Euclidean spaces, we mean finite dimensional Euclidean spaces over the reals. 
\mybreak
Dealing with random vectors, the Borel sets we are interested in have the form $b+C$, where $b\in V$ and $C\subseteq V$ is a non-empty set, often a \emph{cone}. A \emph{cone} in $\Euc$ is a subset $C$ of $V$ that is closed under taking positive scalar multiples, i.e., $\lambda C\subseteq C$ for every $\lambda>0$, while a \emph{convex cone} $C$ is a cone that is closed under taking sum, i.e, $C+C\subseteq C$. Clearly, the set consisting only of the zero vector of $V$, is trivially a cone (the \emph{trivial cone} indeed). In the sequel, when we speak of cones we mean non-empty non-trivial cones. Crucial to the definition of \emph{generalized inequality} is the following notion of duality of subsets of Euclidean spaces. First identify the algebraic dual $V^*$ of $V$ with $V$ by the inner product in $\Euc$ via the isomorphism $V\ni v\mapsto \inp{v,\cdot}\in V^*$. Let now $C$ be a non-empty subset of $V$. The \emph{dual} $C^*$ of $C$ in $\Euc$ is the set
$$C^*=\{u\in V \ |\ \langle u,v\rangle\geq 0,\quad \forall v\in C \}.$$
Write $C^{**}$ for the double dual of $C$, namely $(C^*)^*$.\,The following facts, the first two of which are straightforward consequences of the definition, are known about the dual of a non-empty set $C$ (see \cite{MonBe,schne}).
\begin{enumerate}[label=\textrm{(\alph*)},ref=\alph*]
\item\label{com:c1} $C^*$ is always a closed convex cone;
\item $C\subseteq C^{**}$;
\item\label{com:c3} $C=C^{**}$ if and only if $C$ is a closed convex cone. 
\end{enumerate}
The last property implies that $C^*=C^{***}$ because $C^*$ is a closed convex cone, by \eqref{com:c1}. A cone $C$ is \emph{proper} whenever $C$ is a closed convex cone that is also pointed, i.e., $C\cap -C=\{0\}$, where $0$ is the zero vector of $V$, and has non-empty interior. A cone $C$ is \emph{self-dual in $\Euc$} if $C=C^*$.
\mybreak 
For a random vector $X=(X_1,\ldots,X_n)'\in\R^n$ and $b=(b_1,\ldots,b_n)'\in \R^n$, the event $(X\geq b)$ is said to be a \emph{tail} of $X$. Such an event reads as $(X_1\geq b_1,\ldots,X_n\geq b_n)$ and is the same event as $X-b\in \R^n_+$. The non-negative orthant $\R^n_+$ is a self-dual cone in $\Rn$. Clearly, $X-b$ is non-negative if and only if $u'(X-b)$ is non-negative for all non-negative vectors $u\in \R^n$, i.e., for all $u$ in the dual cone of $\R^n_+$. This fact can be generalized as follows. Let the ambient space of $\Euc$ be $V$ and let $C$ be a non-empty subset of $V$. For $x,\,y\in V$, write $y\succeq_C x$ if $y-x\in C$. If $C$ is a convex cone, then $\succeq_C$ is a pre-order on $V$, while if $C$ is a proper cone, then $\succeq_C$ is a partial order on $V$. In any case, even when $C$ is arbitrary, by duality, one has
\begin{equation}\label{eq:gen_in}
	y\succeq_C x\Longrightarrow y\succeq_{C^{^{**}}} x\Longleftrightarrow \inp{u,y}\geq \inp{u,x},\forall u\in C^*
\end{equation}
which reduces to 
$$y\succeq_C x\Longleftrightarrow \inp{u,y}\geq \inp{u,x},\forall u\in C^*$$
when $C$ is a closed convex cone.
Generalized inequalities, and hence generalized tails, are well behaved with respect to natural transformations of random vectors because cones are preserved by linear transformation. Actually, when the transformation is invertible, closedness is preserved as well. Recall that given a linear map $f$ between the ambient vector spaces $V$ and $W$ of two Euclidean spaces $\Euc$ and $\mathcal{F}$, the \emph{adjoint map of $f$} is the unique linear map $f^*$ that satisfies  $\inp{f(v),w}=\inp{v,f^*(w)}$ for all $v\in V$ and all $w\in W$, where the inner product on left-hand side is the inner product of $\mathcal{F}$, while the inner product on the right-hand side is the inner product of $\Euc$. Moreover, when $f$ is a vector space isomorphism, the dual of the linear image of $C$ can be described  very easily as follows
\begin{equation}\label{eq:lin_image}
	f(C)^*=(f^*)^{-1}(C^*).
\end{equation}
 Indeed, 
$$u\in f(C)^*\Leftrightarrow \inp{f(v),u}\geq 0,\, \forall v\in C\Leftrightarrow \inp{v,f^*(u)}\geq 0,\, \forall v\in C\Leftrightarrow f^*(u)\in C^*\Leftrightarrow u\in (f^*)^{-1}(C^*).$$
We need also a less known duality device which we borrow from the Theory of \emph{Blocking pairs of polyhedra} \cite{biblelp}. Let $V$ be the ambient space of $\Euc$. For $T\subseteq V$, the \emph{blocker of $T$ in $\Euc$} is the set
$$\mathfrak{b}(T)=\{u\in V \ |\ \inp{u,x}\geq 1,\quad \forall x\in T\}.$$  
Analogous to the dual of $T$, the blocker of $T$ has the following properties, the first two of which are straightforward.
\begin{enumerate}[label=\textrm{(\roman*)},ref=\roman*]
\item\label{com:b1} if $\bl{T}$ is non-empty, then $\bl{T}$ is a closed convex set being the intersection of closed half-spaces;
\item\label{com:b2} if $T\subseteq T'$, then $\bl{T}\supseteq \bl{T'}$; moreover, if $\bl{T}$ is non-empty, then $T\subseteq \bl{\bl{T}}$;
\item\label{com:b3} if $T=b+C$ where $b\in V$ and $C$ is such that $\emptyset\not=C\subseteq V$, then 
\begin{equation}\label{eq:chain1}
\bl{b+C}\supseteq \bl{b+C^{**}}\supseteq C^*\cap \{u\in V \ |\ \inp{u,b}\geq 1\};
\end{equation}
moreover, if $C$ is a nontrivial cone, then
\begin{equation}\label{eq:chain2}
C^*\cap \{u\in V \ |\ \inp{u,b}>0\}\supseteq \bl{b+C}=C^*\cap \{u\in V \ |\ \inp{u,b}\geq 1\}.	
\end{equation} 
To prove \eqref{eq:chain1}, observe in the first place if $u\in C^*\cap \{u\in V \ |\ \inp{u,b}\geq 1\}$, then $\inp{u,b+y}\geq 1,\,\forall y\in C^{**}$. Hence $u\in \bl{b+C^{**}}$ and $u\in \bl{b+C}$ because $\bl{b+C}\supseteq \bl{b+C^{**}}$ by \eqref{com:b2}. Let us prove \eqref{eq:chain2}. Since $C^*\cap \{u\in V \ |\ \inp{u,b}\geq 1\}\subseteq C^*\cap \{u\in V \ |\ \inp{u,b}>0\}$, it follows that to prove \eqref{eq:chain2} it suffices to prove that $\bl{b+C}=C^*\cap \{u\in V \ |\ \inp{u,b}\geq 1\}$ which, after \eqref{eq:chain1}, will follow by $\bl{b+C}\subseteq C^*\cap \{u\in V \ |\ \inp{u,b}\geq 1\}$. Let us prove the latter inclusion. Observe that if $u\in \bl{b+C}$ and $C$ is a cone, then necessarily $\inp{u,y}\geq 0$ for all $y\in C$ for, if not, there is $y_0\in C$ such that $\inp{u,y_0}< 0$ and $b+\lambda y_0\in T$ for all $\lambda>0$. Hence $\inp{u,b+\lambda y_0}<1$ for large enough $\lambda$. Thus $\bl{b+C}\subseteq C^*$. By the same reasoning, it holds that $u\in \bl{b+C}\Rightarrow \inp{u,b}\geq 1$. To see this, assume by contradiction that $\inp{u_0,b}<1$ for some $u_0\in \bl{b+C}$. Since $u_0\in \bl{b+C}$, there exists $y_0\in C$, $y_0\not=0$, such that $\inp{u_0,b+\lambda y_0}\geq 1$ for all $\lambda>0$. Nevertheless, the latter inequality cannot be satisfied for a small enough $\lambda$. We conclude that the desired inclusion is true. 
\item\label{com:b4} If $T=P+\R^n_+$ where $P$ is a polytope contained in $\R^n_+$, then $T=\bl{\bl{T}}$ in $\Rn$. Moreover, if $b_1\ldots,b_m$ are the  vertices of $P$, then $\bl{T}=\{u\in \R^n_+ \ |\ Bu\geq \mathbf{1}_m\}$ where $B$ is the matrix whose $i$-th row is $b_i'$, $i=1,\ldots, m$, and $\mathbf{1}_m$ is the all ones vector of $\R^m$. Furthermore, $T$ is of the form $T=\{x\in \R^n_+\ |\ Ax\geq \mathbf{1}_l\}$ where each vertex of $\bl{T}$ occurs among the rows of $A$ and there are most $l$ such vertices. Conversely, if $A$ is any non-negative matrix with $l$ rows and $n$ columns, then the set $T=\{x\in \R^n_+\ |\ Ax\geq \mathbf{1}_l\}$ has the form $P+\R^n_+$ where $P$ is a polytope contained in $\R^n_+$ and the  vertices of $\bl{T}$ all occur among the rows of $A$. See \cite{biblelp} for more details. 
\end{enumerate} 

\begin{remark}
Usually the blocker of a polyhedron $T$ of the form $P+\R^n_+$ is defined as the set $B(T)=\{u\in\R^n_+ \ |\ u'x\geq 1,\, \forall x\in T\}$. Hence $B(T)=\bl{T}\cap \R^n_+$ in $\Rn$. However the two definitions coincide in this case because, reasoning as in the proof of \eqref{eq:chain2}, $\bl{T}\subseteq (\R^n_+)^*=\R^n_+$. Therefore $\bl{T}=B(T)$.  
\end{remark}
In the Euclidean space $\Euc=(V,\inp{\cdot,\cdot})$, for a cone $C$ and a vector $b$ in $V$ the set 
$$\al{b,C}=C^*\cap \{u\in V \ |\ \inp{u,b}>0\}$$
plays a distinguished role in the following intermediate results. 
\begin{lemma}\label{lemma:0}
For a cone $C$ in the Euclidean space $\Euc=(V,\inp{\cdot,\cdot})$, one has $\al{b,C}\not=\emptyset$ if and only if $\bl{b+C}\not=\emptyset$. Moreover, one has $\al{b,C}\not=\emptyset$ if and only if $0\not\in b+C^{**}$.   
\end{lemma}
\begin{proof}
By \eqref{com:b3}, $\bl{b+C}\not=\emptyset\Rightarrow \al{b,C}\not=\emptyset$. Conversely, if $u\in \al{b,C}$, then $\frac{u}{\inp{u,b}}\in \bl{b+C}$. To prove the other assertion observe that 
$$\al{b,C}=\emptyset\Longleftrightarrow \inp{u,b}\leq 0\,\forall u\in C^*\Longleftrightarrow -b\in C^{**}\Longleftrightarrow 0\in b+C^{**}.$$    
\end{proof}
Notice that the condition $0\not\in b+C^{**}$ in Lemma \ref{lemma:0} can be written as $b\in V\setminus -C^{**}$.
\begin{lemma}\label{lemma:1}
In the Euclidean space $\Euc=(V,\inp{\cdot,\cdot})$ let $C$ be a nontrivial cone, and $b\in V\setminus -C^{**}$. If $f,\,g:V\rightarrow \R$ are defined by
$$f(u)=\frac{q(u)}{\inp{u,b}^2+q(u)}\quad\text{and}\quad  g(u)=\frac{q(u)}{1+q(u)}$$
where $q: V\rightarrow \R$ is a positive definite quadratic form, then 
$$\inf_{\al{b,C}}f(u)=\inf_{\bl{b+C}}g(u).$$
\end{lemma}
\begin{proof}
The assumption on $b$ guarantees that $\al{b,C}$ and $\bl{b+C}$ are both non-empty by Lemma \ref{lemma:0}. Since $q$ is a quadratic form, it is homogeneous of degree 2. Hence, $f$ is homogeneous of degree 0, namely $f(\lambda u)=f(u)$. Observe that $f(u)\leq 1$ for all $u\in V$ and that $f(u)=1$ over the hyperplane $\{u\in V \ |\ u'b=0\}$. Hence, by homogeneity, for every $u\in C^*$, $f$ is constant on the rays $\lambda u$, $\lambda>0$. In particular, 
\begin{equation}\label{eq:ancella}
f(u)=f\left(\frac{u}{\inp{u,b}}\right)=g\left(\frac{u}{\inp{u,b}}\right),\quad \forall u\in C^*\cap \{u\in V \ | \ \inp{u,b}>0\}.
\end{equation}
Since by \eqref{eq:chain2} it holds that $\bl{b+C}=C^*\cap \{u\in V \ |\ \inp{u,b}\geq 1\}$, it follows that 
$$f(u)=\frac{q(u)}{\inp{u,b}^2+q(u)}\leq \frac{q(u)}{1+q(u)}=g(u).$$
Hence
\begin{equation}\label{eq:ancella1}
	f(u)\leq g(u)\quad \forall u\in \bl{b+C}.
\end{equation}
Let $H^*=\{u\in V \ | \ \inp{u,b}=1\}$ and observe that
\[
\begin{split}
\inf_{\al{b,C}}f(u)&=\inf_{\al{b,C}\cap H^*}g(u)\quad (\text{by\eqref{eq:ancella}})\\
	&\geq\inf_{\bl{b+C}}g(u)\quad (\text{because $\al{b,C}\cap H^*\subseteq \bl{b+C}$ by \eqref{eq:chain2}})\\
	&\geq\inf_{\bl{b+C}}f(u)\quad (\text{by \eqref{eq:ancella1}})\\
	&\geq\inf_{\al{b,C}}f(u)\quad (\text{because $\al{b,C}\supseteq \bl{b+C}$ by \eqref{eq:chain2}})
	\\
\end{split}
\]
therefore equality must hold throughout yielding the desired equality.
\end{proof}
Recall that given a self-adjoint positive definite endomorphism of $V$ in $\Euc=(V,\inp{\cdot,\cdot})$, the bilinear form $\inp{u,A(v)}$ induces a norm $p$ on $V$ by $p(v)=\sqrt{\inp{v,A(v)}}$. The \emph{dual norm} of $p$ is the norm $p^*$ such that $$p^*(v)=\sup_{u\not=0}\frac{\inp{u,v}}{\sqrt{\inp{u,A(u)}}}.$$ 
One has $p^*(v)=\sqrt{\inp{v,A^{-1}(v)}}$. Although this fact is \virg{folklore}, we give a proof here due to lack of references.

Since $A$ is a self-adjoint positive definite endomorphism, $A=B^2$ for some self-adjoint positive definite endomorphism $B$ ($B$ is a \emph{square-root} of $A$). Hence, by the Cauchy-Schwarz inequality and because $B$ is self-adjoint,
$$\inp{u,v}=\inp{u,BB^{-1}(v)}=\inp{B(u),B^{-1}(v)}\leq \sqrt{\inp{B(u),B(u)}}\sqrt{\inp{B^{-1}(v),B^{-1}(v)}}=\sqrt{\inp{u,A(u)}}\sqrt{\inp{v,A^{-1}(v)}}.$$
Therefore $p^*(v)\leq \sqrt{\inp{v,A^{-1}(v)}}$ with equality for $v=0$. On the other hand, if $v\not=0$, then the vector $\overline{v}$ defined by
$$\overline{v}=\frac{A^{-1}(v)}{\sqrt{\inp{v,A^{-1}(v)}}}$$
is such that 
$$p^*(v)\geq \frac{\inp{\overline{v},v}}{\sqrt{\inp{\overline{v},A(\overline{v})}}}=\sqrt{\inp{v,A^{-1}(v)}}.$$
Hence $p^*$ has the stated expression. 
\begin{lemma}\label{lemma:2}
In the Euclidean space $\Euc=(V,\inp{\cdot,\cdot})$ let $C$ be a closed convex cone, $b\in V\setminus -C$, and $p$ the norm $p(v)=\sqrt{\inp{v,A(v)}}$ on $V$ where $A$ is a self-adjoint positive definite endomorphism of $V$. If $b\in A(C^*)$ then, with $q=p^2$, in  the notation of Lemma \ref{lemma:1},
$$\inf_{\al{b,C}}f(u)=\frac{1}{(p^*(b))^2+1}$$
where $p^*$ is the dual norm of $p$, namely
$$p^*(b)=\sup_{u\not=0}\frac{\inp{u,b}}{p(u)}=\sqrt{\inp{b,A^{-1}(b)}}.$$ 
\end{lemma}
\begin{proof}
Since $C$ is a closed convex cone, then $C=C^{**}$ by \eqref{com:c3}. Hence, the assumption on $b$ guarantees that $\al{b,C}$ is non-empty by Lemma \ref{lemma:0}. Therefore, we can divide the numerator and the denominator of $f(u)$ by $q(v)$. This yields    
$$\inf_{\al{b,C}}f(u)=\left(\sup_{\al{b,C}}\left\{\frac{\inp{u,b}}{p(u)}\right\}^2+1\right)^{-1}.$$
Now, if $b\in A(C^*)$, then $A^{-1}(b)\in C^*$ by \eqref{eq:lin_image}. Hence $\hat{b}=\frac{A^{-1}(b)}{\sqrt{\inp{b,A^{-1}(b)}}}$ belongs to $\al{C,b}$ because $\hat{b}\in C^*$, $\hat{b}$ being a positive scalar multiple of $A^{-1}(b)$, and $\inp{\hat{b},b}= \sqrt{\inp{b,A^{-1}(b)}}>0$. After plugging $\hat{b}$ into $\frac{\inp{u,b}}{p(u)}$, we conclude that the supremum $p^*(b)$ of $\frac{\inp{u,b}}{p(u)}$ is attained over $\al{b,C}$ and this concludes the proof. 
\end{proof}

\section{Main results}\label{sec:main}
In order to state and prove the main results of the paper, we shall recap a very (few) basic facts about random vectors. We follow \cite{eaton}. Let $(\Omega,\mathcal{F},\mathbb{P})$ be a probability space, where $\Omega$ is a set, $\mathcal{F}$ is a $\sigma$-algebra on $\Omega$, and $\mathbb{P}$ is a probability measure on $\mathcal{F}$. Also let $\Euc=(V,\inp{\cdot,\cdot})$ be an Euclidean space and $\mathcal{B}$ the smallest $\sigma$-algebra containing all the open balls of $V$ taken with respect to the norm induced by $\inp{\cdot,\cdot}$---this $\sigma$-algebra does not depend on the the particular inner product chosen on $V$---. A \emph{random vector $X$ in $\Euc$}, is a $(\mathcal{F},\mathcal{B})$-measurable map $X:\Omega\rightarrow V$. The algebra $\mathcal{B}$ is the \emph{algebra of Borel sets of $\Euc$} and of course contains generalized tails as Borel sets. The (\emph{induced}) distribution $Q$ of $X$ is the probability measure on $\mathcal{B}$ defined by $Q(B)=\mathbb{P}(X^{-1}(B))$ for each $B\in \mathcal{B}$. If $\Euc_1=(V_1,\inp{\cdot,\cdot}_1)$ is another Euclidean space, a map $f:V\rightarrow V_1$ is \emph{Borel measurable} if $f$ is $(\mathcal{B},\mathcal{B}_1)$-measurable with $\mathcal{B}_1$ being the algebra of Borel sets of $\Euc_1$. If, $f$ is $(\mathcal{B},\mathcal{B}_1)$-measurable, then $f(X)$ is a random vector in $\Euc_1$. The \emph{mean vector} and the \emph{covariance} of the random vector $X$ of $\Euc$ are defined as follows. The map $f: V\rightarrow \R$ defined by $f(v)=\Ea{\inp{v,X}}$, where $\Ea{\cdot}$ denotes the expectation with respect to the distribution of $X$, is readily seen to be linear, provided that the expectation of $\inp{v,X}$ exists $\forall v\in V$. Hence $f\in V^*$ and, by Riesz's Representation Theorem, there is a unique vector $\mu\in V$ such that $f=\inp{\cdot,\mu}$. Such a vector is the mean of $X$. We are interested in \emph{centered random vectors}, namely, those vectors whose mean vector is zero ($f$ is the null vector of $V^*$). It can be shown that if $Y$ in $\Euc_1$ is the image of a centered random vector $X$ in $\Euc$ under a linear map, then $Y$ is a centered random vector. Now, let $X$ be a centered random vector in $\Euc$ and consider the map $F:V\times V\rightarrow \R$, $(u,v)\mapsto \Ea{\inp{u,X}\inp{v,X}}$. Since $F(u,u)$ is the variance of the random variable $\inp{u,X}$, if $F(u,u)$ is finite for all $u\in V$, then it is readily seen that $F$ is a symmetric positive semi-definite bilinear form in $\Euc$. Therefore, there exists a unique positive semi-definite self-adjoint endomorphism $\Sigma$ of $V$ such that $F(u,v)=\inp{u,\Sigma(v)}$. Such an endomorphism is the \emph{covariance} of $X$. If $L$ is a linear map between Euclidean spaces $\Euc$ and $\mathcal{F}$, and $X$ is a random vector in $\Euc$, then $L\Sigma L^*$ is the covariance of $L(X)$. In particular, if $L$ is an isomorphism of vector spaces, then $L\Sigma L^*$ is positive definite if and only if $\Sigma$ is positive definite.
When $\Euc=\Rn$, we recover the notion of random vector as a multivariate random variable with zero mean and positive semi-definite covariance matrix $\Sigma$: $\Sigma=\left\{F(e_i,e_j)\right\}_{i,j}$, where $e_i$ is the $i$-th vector of the standard orthonormal basis of $\Rn$. From now on, we assume that random vectors have a positive definite covariance, i.e., $F(u,u)>0$, for all $u\in V$, therefore the covariance of any of their images under an invertible linear map is such. We are now in position to state and prove our main results.

\begin{theorem}\label{thm:mar_olk_gen}
Let $X$ be a centered random vector in $\Euc=(V,\inp{\cdot,\cdot})$ with positive definite covariance $\Sigma$. Let $T$ be a non-empty subset of $V$ and $b\in V$, $b\not=0$. Then, for all $u\in \bl{T}$, it holds that 
$$\pra{X\in T}\leq \frac{\inp{u,\Sigma(u)}}{1+\inp{u,\Sigma(u)}}.$$
Therefore, if $\bl{T}$ is non-empty, then
\begin{equation*}
\pra{X\in T}\leq \inf_{u\in \bl{T}}\frac{\inp{u,\Sigma(u)}}{1+\inp{u,\Sigma(u)}}.
\end{equation*}
If $T$ is a closed convex set such that  $0\not\in T$, then the bound above is sharp.
\end{theorem}
\begin{proof}
The second inequality follows from the first provided that $\bl{T}$ is non-empty. The proof of the first inequality is formally identical to the proof of the same inequality in $\Rn$ given in Section \ref{sec:intro} and is a direct consequence of \eqref{com:b2} and Cantelli's inequality \eqref{eq:can_cla} after noticing that for a random vector $X$ in $\Euc$ the variance of the random variable $\inp{u,X}$ is $\inp{u,\Sigma u}$. It remains to show that if $T$ is closed and convex in $\Euc$ and $0\not\in T$, then the bound is sharp. We deduce this result from Theorem \ref{thm:mar_olk} after reducing the general case to $\Rn$ by the following argument. Every Euclidean space is a topological vector space with respect to the standard topology induced by the inner product (recall that, in our terminology, Euclidean spaces are finite dimensional, and hence Hilbert spaces). Euclidean spaces of the same dimension are pairwise homeomorphic and all are homeomorphic to $\Rn$ under the coordinate map isomorphism $f$ (with respect to a fixed orthonormal basis). Hence, if $X$ is a centered random vector in $\Euc$, then $f(X)$ is a centered random vector in $\Rn$, and conversely. Moreover, if the covariance $\Sigma$ of $X$ is positive definite, then the covariance matrix $\widetilde{\Sigma}$ of $f(X)$ is positive definite, and conversely. Furthermore, $T$ is a convex closed set in $\Euc$ if and only if $f(T)$ is a closed set in $\Rn$ (with the standard Euclidean topology): linear images of convex sets are convex, and homeomorphic linear image of closed convex sets are closed and convex. Finally, since $f$ maps the zero vector $0_V$ of $\Euc$ into the zero vector $0$ of $\Rn$, it follows that $0_{V}\not\in T$ if and only if $0\not\in f(T)$. Now 
$$\pra{X\in T}=\pra{f(X)\in f(T)}.$$
Therefore, if $\widetilde{X}_0$ is a random vector in $\Rn$ with positive definite matrix $\widetilde{\Sigma}$ attaining the bound in Theorem \ref{thm:mar_olk} as equality, which exists because the hypotheses on $f(T)$ are satisfied, then $f^{-1}(\widetilde{X}_0)$ is a random vector in $\Euc$ attaining the bound in the present theorem. 
\end{proof}

\begin{theorem}\label{thm:mar_olk_gin} Let $X$ be a centered random vector of $\Euc=(V,\inp{\cdot,\cdot})$ with positive definite covariance $\Sigma$. Let $C$ be a non-empty subset of $V$ and $b\in V$, $b\not=0$. Then, for all $u\in C^*$ such that $\inp{u,b}>0$, it holds that  
$$\pra{X\succeq_C b}\leq \frac{\inp{u,\Sigma(u)}}{\inp{u,b}^2+\inp{u,\Sigma(u)}}.$$
If $C$ is a closed convex cone in $\Euc$ and $b\in V\setminus -C$, then $\al{b,C}$ is non-empty and
\begin{equation*}
\pra{X\succeq_C b}\leq \inf_{u\in \al{b,C}} \frac{\inp{u,\Sigma(u)}}{\inp{u,b}^2+\inp{u,\Sigma(u)}}=\inf_{u\in \bl{b+C}}\frac{\inp{u,\Sigma(u)}}{1+\inp{u,\Sigma(u)}}.
\end{equation*}
Moreover, the bound above is sharp.
\end{theorem}
\begin{proof}
Since $\pra{X\succeq_C b}\leq \pra{X\succeq_{C^{^{**}}} b}$, the first inequality follows from \eqref{eq:gen_in} (with $X$ in place of $y$ and $b$ in place of $x$) by applying Cantelli's inequality \eqref{eq:can_cla} to the random variable $\inp{u,X}$. 
If $C$ is a closed convex cone and $b\in V\setminus -C$, then $\al{b,C}$ is non-empty because the hypotheses of Lemma \ref{lemma:0} are satisfied (recall that since $C$ is a closed convex cone one has $C=C^{**}$). Moreover, by the same lemma, $\bl{b+C}$ is non-empty as well. Hence, the second inequality follows from the first by Lemma \ref{lemma:1} with $q(u)=\inp{u,\Sigma(u)}$. It remains to prove that the bound is sharp. Let $T=b+C$. Since $C$ is closed and convex, then so is $T$. Moreover, the hypotheses on $b$ and $C$ implies $0\not\in T$. Therefore, Theorem \ref{thm:mar_olk_gen} applies and we conclude that the bound is sharp.
\end{proof}
Although the previous theorem deals with the special case of $T=b+C$ of Theorem \ref{thm:mar_olk_gen} which in turn, up to technicalities, follows by Theorem \ref{thm:mar_olk}, Theorem \ref{thm:mar_olk_gin} provides a useful sharpening of the extended Cantelli's inequality given in Theorem \ref{thm:mar_olk} which is further exploited in the next corollary. We have already noticed in fact that the minimization problem involved in all the bounds above, reduces to minimizing the positive definite quadratic form $q(u)=\inp{u,\Sigma(u)}$ over $\al{b,C}$ or over $\bl{T}$. Even when $T=b+C$ and $C$ is a convex cone, this minimization problem is a non trivial convex programming problem \cite{nem}. However, settling for looser estimates, we can easily compute useful upper bounds on the probability of certain interesting tail events. For instance, $C$ can be the cone of co-positive  matrices, namely those symmetric matrices $A$ such that $x'Ax\geq 0$ for all $x\in \R^n_+$, or the sub-cone of positive semi-definite matrices \cite{MonBe}. In the latter case, $C$ is a self-dual cone and, if $b$ is a symmetric positive definite matrix, then the probability that $X-b$ is a positive semi-definite matrix, where $X$ is a random symmetric matrix, is the same as $\pra{X\succeq_{S^n_+} b}$ where $S^n_+$ denotes the cone of positive-definite matrices sitting in $\Euc=(S^n,\inp{\cdot,\cdot})$, with $S^n$ being the real vector space of the real symmetric matrices of order $n$, and $\inp{u,v}=\text{trace}(u\cdot v)$ being the Frobenius inner product. Suppose that $X$ is sampled from a centered and \emph{weakly spherical distribution}, namely, a distribution such that $\Sigma=\sigma^2 I$, where $I$ is identity endomorphism of $V$ and $\sigma^2$ is a positive real number. This includes the case in which  the entries of $X$ are sampled independently from the same centered distribution with variance $\sigma^2$ \cite{eaton}. Now, since $b\in S^n_+$ (because it is definite positive by hypothesis) and $\inp{b,\Sigma b}=\sigma^2\inp{b,b}=\sigma^2\text{trace}(b^2)>0$, Theorem \ref{thm:mar_olk_gin} applies and  
$$\pra{X\succeq_{S^n_+} b}\leq\frac{\sigma^2\text{trace}(b^2)}{(\text{trace}(b^2))^2+\sigma^2\text{trace}(b^2)}=\frac{\sigma^2}{\text{trace}(b^2)+\sigma^2}.$$
If $b=\lambda I_n$, where $I_n$ is the real identity matrix of order $n$ and $\lambda$ is a positive real number, then $\pra{X \succeq_{S^n_+} \lambda I_n}$ has a decay in $n$ which is not slower than $O(n^{-1})$. Hence, the probability that $\lambda$ bounds from below the least eigenvalue of a random symmetric matrix $X$ of order $n$ with centered weakly spherical distribution tends to zero (in $n$) with order not slower than $n^{-1}$. Therefore, if $E(\lambda)$ is the event that a symmetric matrix has its smallest eigenvalue bounded from below by $\lambda$, then for each $\alpha<1/2$, $E({n^{-2\alpha}})$ is an increasing sequence of events whose limit is $E(0)$ and 
\[
\lim_{n\rightarrow \infty}\pra{X \succeq_{S^n_+} 0}=\lim_{n\rightarrow \infty}\pra{E({n^{-2\alpha}})}\leq \lim_{n\rightarrow \infty}\frac{\sigma^2}{n^{1-2\alpha}+\sigma^2}=0.	
\] 
As one can expect, we conclude that the probability of sampling positive semi-definite matrices of order $n$ from symmetric matrices distributed according to a centered weakly spherical distribution, is zero in the limit. These facts are consequences of the special form of the covariance and of the special choice of $b$. By abstracting these properties, the argument can be generalized as follows.
\begin{corollary}\label{cor:mahala} Let $X$ be a centered random vector of $\Euc=(V,\inp{\cdot,\cdot})$ with positive definite covariance $\Sigma$. Let $C$ be a closed convex cone in $\Euc$. If $b\in \Sigma(C^*)\setminus -C$, where $\Sigma(C^*)$ denotes the linear image of $C^*$ under $\Sigma$, then
\begin{equation}\label{eq:mine_2}
\pra{X\succeq_C b}\leq \frac{1}{1+\nmaha{b}^2}
\end{equation}
where $\nmaha{b}=\inp{b,\Sigma^{-1}(b)}$. Moreover, the bound is sharp.
\end{corollary}
\begin{proof}
Directly from Lemma \ref{lemma:1} and Lemma \ref{lemma:2} with $p^*=\sqrt{\inp{b,\Sigma^{-1}(b)}}$.
\end{proof}
As another application of the corollary, observe that the hypotheses of Corollary \ref{cor:mahala} are certainly met in $\Rn$ when $C=\R^n_+$ and $X$ is a centered random vector in $\Rn$ whose covariance matrix $\Sigma$ has a non-negative inverse: $\Sigma^{-1}\geq 0$---this includes the case of random vectors with uncorrelated coordinates---. Indeed, if $C=C^*=\R^n$ and $\Sigma^{-1}\geq 0$, then $\Sigma^{-1}(C^*)\subseteq C^*$. Hence, for all $b\in \R^n_+$, $\pra{X\geq b}\leq (1+\|b\|^2_M)^{-1}$, where $\|b\|_M=\|b\|_{\Sigma^{-1}}$ is the so-called Mahalanobis norm of $b$. Therefore, in this case, we have a useful and handy tool for testing statistical hypotheses. To see this, denote by $\R^n_{++}$ the interior of $\R^n_+$ and let $Y$ be a random (almost surely) non-negative vector in $\Rn$ whose positive definite covariance matrix $\Sigma$ has non-negative inverse. Let $\mu$ the mean vector of $Y$. Clearly $\mu$ is non-negative. Suppose $\mu\in \R^n_{++}$ and let $\lambda$ be a positive real number. Hence 
$$\pra{Y\geq (1+\lambda)\mu}=\pra{Y-\mu\geq \lambda\mu}\leq (1+\lambda^2\|\mu\|^2_{\Sigma^{-1}})^{-1}$$
and this yields a conservative criterion to test whether an observed value $y$ of $Y$ is significantly large: 
\begin{itemize}
\item[--] fix a significance level $\alpha$ (typically $\alpha=5\times 10^{-2}$);
\item[--] let $$\lambda=\frac{1}{|\mu\|^2_{\Sigma^{-1}}}\sqrt{\frac{1-\alpha}{\alpha}};$$
\item[--] if  $\lambda\leq \min_i\frac{y_i}{1+\mu_i}$, then declare $y$ large at the significance level $\alpha$; otherwise, declare $y$ not large at the same significance level.  
\end{itemize}
If $Y$ is not necessarily non-negative but the observed value $y$ is such that $y-\mu\in \R^n_{++}$, then the following simple measure allows us to evaluate how significantly far is $y$ from the mean. Since
$$\pra{Y\geq y}=\pra{Y-\mu\geq y-\mu}\leq \frac{1}{1+\|y-\mu\|^2_{\Sigma^{-1}}}$$
we conclude that the Mahalanobis norm of $y-\mu$ is a directed measure of the statistical significance of the deviation of $y$ from $\mu$. When $\Sigma$ has additional structure, the analysis can be refined. Suppose that  $\Sigma=D-\gamma uu'$, with $D$ a positive definite diagonal matrix, $u\in \R^n_+$, and $\gamma>0$. Hence $\Sigma$ is a rank-one update of a diagonal matrix. Such matrices arise when each coordinate of a random vector $X$ is a function of the elements of a class of a partition of a given ground set $S$ and there is a negative interaction between elements in different pairs of classes. If $\Sigma$ has this form, then, by the Sherman–Morrison formula, not only $\Sigma$, but also all of its principal minors have non-negative inverse. Now, let $J$ and $K$ be complementary subsets of $\{1,2,\ldots,n\}$, and let $W=\bigoplus_K\R e_i$, where $e_i$ is the $i$-th fundamental vector of $\R^n$, $i=1\ldots,n$. Hence, $W$ is the subspace spanned, say, by the coordinates whose index is in $K$. We can then project $X$ onto $\R^n/ W$. By linearity, the mean vector of the projection is simply the mean vector of $X$ with the coordinates whose index is in $K$ suppressed, while the covariance matrix of the projection is the principal minor of $\Sigma$ defined by $J$. Therefore, repeating the previous argument with $J$ varying among small sized subsets of $\{1,2,\ldots,n\}$, we may identify subsets of coordinates that are most responsible for the deviation.
\mybreak
An interesting consequence of Theorem \ref{thm:mar_olk} in $\Rn$, comes straightforwardly from another kind of duality of convex sets, namely, the Theory of Blocking Polyhedra (recall \eqref{com:b4} in Section \ref{sec:prel}). Consider the set $T\subseteq \R^n$ defined by $T=\{x\in \R^n \ |\ Ax\geq b, x\geq 0\}$ where $A$ is a non-negative real matrix with $m$ rows and $n$ columns and $b$ is a non-negative vector of $\R^m$. If $T$ is non-empty, then $T$ is a closed convex set (a polyhedron, in fact). Moreover, if $b\not=0$, then $0\not\in T$ and, by Theorem \ref{thm:mar_olk}, the probability $\pi(X,A,b)$ that a centered random vector in $\Rn$ satisfies the system of linear inequalities  $Ax\geq b, x\geq 0$ is bounded from above by $\inf_{u\in\bl{T}}f(u)$, where, as in Lemma \ref{lemma:1}, we have set $f(u)=u'\Sigma u/(1+u'\Sigma u)$. Suppose further that $b\in \R^m_{++}$. After scaling $A$, it follows that $X\in T$ if and only if $X\in \widetilde{T}$, where $\widetilde{T}=\{x\in \R^n \ |\ \widetilde{A}x\geq \mathbf{1}_m, x\geq 0\}$, $\widetilde{A}$ being the scaled matrix. Let $\widetilde{H}$ denote the convex hull of the rows of $\widetilde{A}$ viewed as vectors of $\R^m$. Since $\bl{\widetilde{T}}=\widetilde{H}+\R^m_+$ is a polyhedron again (recall \eqref{com:b4}), it follows that bounding $\pi(X,A,b)$ amounts to solve a convex quadratic problem over a polyhedron and each row $\widetilde{a}$ of $\widetilde{A}$ gives the upper bound $f(\widetilde{a})$ on $\pi(X,A,b)$. When $\widetilde{A}$ has a special form, these easy observations can be used profitably. Consider as an illustration the case when $\widetilde{A}$ is the \emph{edge versus vertices incidence matrix} of a graph $G$ with $m$ edges and $n$ vertices, none of which is isolated. Recall that the \emph{edge versus vertices incidence matrix} of $G$ has a row for each edge of $G$, a column for each vertex of $G$ and the entry corresponding to edge $e$ and vertex $v$ is 1 if $v$ and $e$ are incident, and zero otherwise. A \emph{matching} in $G$ is a subset consisting of pairwise non-adjacent edges. Let $\nu(G)$ denote the maximum cardinality of matching of $G$. Suppose that $n$ is even and that $\nu(G)=n/2$ (complete graphs, cycles, paths, all of them  with even order, and complete bipartite graphs with shores of the same cardinality have this property, for instance). Let $M$ be a matching of $G$ with $\nu(G)$ edges. The sum of the rows of $\widetilde{A}$ corresponding to the edges of $M$ equals $\mathbf{1}_n$. Hence $\frac{1}{\nu(G)}\mathbf{1}_n$ is a convex combination of the rows of $\widetilde{A}$ and therefore $\frac{1}{\nu(G)}\mathbf{1}_n\in \bl{\widetilde{T}}$. If the the covariance matrix of $X$ is $\sigma^2I_n$, then 
$$\pi(X,\widetilde{A},\mathbf{1}_m)\leq f\left(\frac{\mathbf{1}_n}{\nu(G)}\right)=\frac{\sigma^2}{4n+\sigma^2},$$
and we conclude that $\pi(X,\widetilde{A},\mathbf{1}_m)$ is $O(n^{-1})$.

\end{document}